\documentclass[12pt]{article}
\usepackage{amsfonts}
\usepackage{amsmath}
\usepackage{amssymb}
\usepackage{graphicx}

\newtheorem{theorem}{Theorem}

\newtheorem{lemma}[theorem]{Lemma}

\newenvironment{proof}[1][Proof]{\textbf{#1.} }{\ \rule{0.5em}{0.5em}}

\input{tcilatex}
\begin{document}

\title{Harmonic vectors and matrix tree theorems}
\author{Siddhartha Sahi \\
Mathematics Department, Rutgers University, New Brunswick, NJ}
\maketitle

\section{Introduction}

In this paper we prove a new result in graph theory that was motivated by
considerations in mathematical economics; more precisely by the problem of
price formation in an exchange economy \cite{DS}. The aggregate
demand/supply in the economy is described by an $n\times n$ matrix $A=\left(
a_{ij}\right) $ where $a_{ij}$ is the amount of commodity $j$ that is on
offer for commodity $i$. In this context one defines a \emph{market-clearing}
price vector to be a vector $p$ with strictly positive components $p_{i}$,
which satisfies the equation%
\begin{equation}
\sum_{j}a_{ij}p_{j}=\sum_{j}a_{ji}p_{i}\text{ for all }i  \label{=price}
\end{equation}%
The left side of (\ref{=price}) represents the total value of all
commodities being offered for commodity $i$, while the right side represents
the total value of commodity $i$ in the market. It was shown in \cite{DS}
that if the matrix $A$ is irreducible, \emph{i.e.} if it cannot be permuted
to block upper-triangular form, then (\ref{=price}) admits a positive
solution vector $p$, which is unique up to a positive multiple.

The primary purpose of the present paper is to describe an explicit
combinatorial formula for $p$. The formula and its proof are completely
elementary, but nonetheless the result seems to be new. This formula plays a
crucial role in forthcoming joint work of the author \cite{DSS}, which seeks
to address a fundamental question in mathematical economics: \emph{How do
prices and money emerge in a barter economy?} We show in \cite{DSS} that
among a reasonable class of exchange mechanisms, trade via a commodity
money, even in the absence of transactions costs, minimizes complexity in a
very precise sense.

It turns out however that equation (\ref{=price}) is closely related to
well-studied problems in graph theory, in particular to the so-called matrix
tree theorems. Therefore as an additional application of our formula, we
give an elementary proof of the matrix tree theorem of W. Tutte \cite{T},
which was independently discovered by R. Bott and J. Mayberry \cite{B}
coincidentally also in an economic context. With a little additional effort,
we also obtain a short new proof of S. Chaiken's generalization of the
matrix tree theorem \cite{C}.

\section{Harmonic vectors}

We first give a slight reformulation and reinterpretation of equation (\ref%
{=price}) in standard graph-theoretic language. Let $G$ be a simple directed
graph (digraph) on the vertices $1,2,\ldots ,n$, with weight $a_{ij}$
attached to the edge $ij$ from $i$ to $j$. The weighted \emph{adjacency}
matrix of $G$ is the $n\times n$ matrix $A=\left( a_{ij}\right) ,$ where $%
a_{ij}=0$ for missing edges. The \emph{degree matrix} $D$ is the diagonal
matrix with diagonal entries $\left( d_{1},\ldots ,d_{n}\right) $, where $%
d_{i}$ is the \emph{in-degree} $\sum_{j}a_{ji}$ of the vertex $i$. The \emph{%
Laplacian} of $G$ is the matrix $L=D-A$ and we say that a vector $\mathbf{x}%
=\left( x_{i}\right) $ is \emph{harmonic} if $\mathbf{x}$ is a null vector
of $L,$ \textit{i.e.} if it satisfies 
\begin{equation}
L\mathbf{x}=\mathbf{0.}  \label{=harmonic}
\end{equation}%
It is easy to see that equation (\ref{=price}) is equivalent to equation (%
\ref{=harmonic}), \textit{i.e. }the market-clearing condition is the same as
harmonicity of $p.$

To describe our construction of a harmonic vector, we introduce some
terminology. A \emph{directed tree}, also known as an \emph{arborescence},
is a digraph with at most one incoming edge $ij$ at each vertex $j$, and
whose underlying undirected graph is acyclic and connected (\textit{i.e.} a
tree). Following the edges backwards from any vertex we eventually arrive at
the same vertex called the \emph{root}. Dropping the connectivity
requirement leads to the notion of a \emph{directed forest,} which is simply
a vertex-disjoint union of directed trees. We define a \emph{dangle} to be a
digraph $D$ that is an edge-disjoint union of a directed forest $F$ and a
directed cycle $C$ linking the roots of $F$; note that $D$ determines $C$
and $F$ uniquely, the former as its unique simple cycle.

In the context of the digraph $G$, we will use the term $i$-tree to mean a
directed spanning tree of $G$ with root $i$, and $i$-dangle to mean a
spanning dangle whose cycle contains $i$. We define the weight $wt\left(
\Gamma \right) $ of a subgraph $\Gamma $ of $G$ to be the product of weights
of all the edges of $\Gamma $, and we define the \emph{weight vector} of $G$
to be $\mathbf{w}=\left( w_{i}\right) $ where $w_{i}$ is the weighted sum of
all $i$-trees.

\begin{theorem}
\label{weight}The weight vector of a digraph is harmonic.
\end{theorem}

\begin{proof}
If $\Gamma $ is an $i$-dangle in $G$ with cycle $C$, and $ij$ and $ki$ are
the unique outgoing and incoming edges at $i$ in $C,$ then deleting one of
these edges from $\Gamma $ gives rise to an $j$-tree and a $i$-tree,
respectively. The dangle can be recovered uniquely from each of the two
trees by reconnecting the respective edges; thus, writing $\mathcal{T}_{i}$
for the set of $i$-trees, we obtain bijections from the set of $i$-dangles
to each of the following sets%
\begin{equation*}
\left\{ \left( ij,t\right) :t\in \mathcal{T}_{j}\right\} ,\quad \left\{
\left( ki,t\right) :t\in \mathcal{T}_{i}\right\} .
\end{equation*}%
where $ij$ and $ki$ range over all outgoing and incoming edges at $i$ in $G$.

Thus if $v_{i}$ is the weighted sum of all $i$-dangles, we get 
\begin{equation*}
\sum\nolimits_{j}a_{ij}w_{j}=v_{i}=\sum\nolimits_{k}a_{ki}w_{i}.
\end{equation*}%
Rewriting this we get $A\mathbf{w}=D\mathbf{w}$, and hence $\left(
D-A\right) \mathbf{w}=\mathbf{0,}$ as desired.
\end{proof}

\section{The matrix tree theorem}

In this section we use Theorem \ref{weight} to derive the \emph{matrix tree
theorem} due to \cite{T} (see also \cite{B}). This is the following formula
for the cofactors of the Laplacian $L$, which generalizes a classical
formula of Kirchoff for the number of spanning trees in an undirected graph.

\begin{theorem}
\label{matrix tree} The $ij$-th cofactor of the Laplacian $L$ is given by%
\begin{equation*}
c_{ij}\left( L\right) =\sum\nolimits_{t\in \mathcal{T}_{j}}wt\left( t\right) 
\text{ for all }i,j\text{.}
\end{equation*}
\end{theorem}

We will prove this in a moment after some discussion on cofactors.

\subsection{Interlude on cofactors}

We recall that $ij$-th cofactor of an $n\times n$ matrix $X$ is 
\begin{equation*}
c_{ij}\left( X\right) =\left( -1\right) ^{i+j}\det X_{ij},
\end{equation*}%
where $X_{ij}$ is the matrix obtained from $X$ by deleting row $i$ and
column $j$. The \emph{adjoint} of $X$ is the $n\times n$ matrix $\limfunc{adj%
}\left( X\right) $ whose $ij$-th entry is $c_{ji}\left( X\right) $.

\begin{lemma}
If $\det X=0$ then the columns of $\limfunc{adj}\left( X\right) $ are null
vectors of $X$; moreover these are the \emph{same} null vector if the
columns of $X$ sum to $0$.
\end{lemma}

\begin{proof}
By standard linear algebra we have $X\limfunc{adj}\left( X\right) =\det
\left( X\right) I_{n}$. If $\det X=0$ then $X\limfunc{adj}\left( X\right) $
is the zero matrix, which implies the first part. For the second part we
note that if $X$ has zero column sums then necessarily $\det X=0.$ In view
of the first part it suffices to show that $c_{ij}\left( L\right)
=c_{i+1,j}\left( L\right) $ for all $i,j$; or equivalently that%
\begin{equation*}
\det \left( L_{ij}\right) +\det \left( L_{i+1,j}\right) =0.
\end{equation*}%
The left side above equals $\det P$, where $P$ is obtained from $L$ by
deleting column $j$ and replacing rows $i$ and $i+1$ by the single row
consisting of their sum. But $P$ too has zero column sums, and so $\det P=0$.
\end{proof}

\subsection{Proof of the matrix tree theorem}

\begin{proof}
It suffices to prove Theorem \ref{matrix tree} for the complete simple
digraph $G_{n}$ on $n$ vertices, with edge weights $\left\{ a_{ij}\mid i\neq
j\right\} $ regarded as variables, and we work over the field of rational
functions $\mathbb{C}\left( a_{ij}\right) $. The Laplacian $L$ has zero
column sums by construction,and so by the previous lemma, $%
c_{j}:=c_{ij}\left( L\right) $ is independent of $i$ and the vector $\mathbf{%
c=}\left( c_{1},\ldots ,c_{n}\right) ^{t}$ is a null vector for $L$. To
complete the proof it suffices to show that the null vectors $\mathbf{c}$
and $\mathbf{w}$ are equal. Now the null space of $L$ is $1$-dimensional
since $c_{ij}\left( L\right) \neq 0$, and hence 
\begin{equation}
c_{i}w_{j}=c_{j}w_{i}\text{ for all }i,j.  \label{=cw}
\end{equation}

Note that $c_{j}$ and $w_{j}$ belong to the polynomial ring $\mathbb{C}\left[
a_{ij}\right] $. We claim that the polynomials $c_{j}$ are distinct and
irreducible. Consider first $c_{n}=\det \left( B\right) $ where $B=L_{nn}$
has entries%
\begin{equation*}
b_{ij}=\left\{ 
\begin{tabular}{cc}
$-a_{ij}$ & if $i\neq j$ \\ 
$a_{nj}+\sum_{k=1}^{n-1}a_{kj}$ & if $i=j$%
\end{tabular}%
\right. ;\quad \text{ for }1\leq i,j\leq n-1.
\end{equation*}%
This is an $\emph{invertible}$ $\mathbb{C}$-\emph{linear} map relating $%
\left\{ b_{ij}\right\} $ to the $\left( n-1\right) ^{2}$ variables 
\begin{equation*}
\left\{ a_{ij}\mid 1\leq i\leq n,1\leq j\leq n-1,i\neq j\right\} ,
\end{equation*}%
which occur in $c_{n}$. Thus the irreducibility of $c_{n}$ follows from the
irreducibility of the determinant as a polynomial in the matrix entries \cite%
[P. 176]{B}. The argument for the other $c_{i}$ is similar, and their
distinctness is obvious.

Since $c_{i}$ and $c_{j}$ are distinct and irreducible, we conclude from (%
\ref{=cw}) that $c_{i}$ divides $w_{i}$. Since $c_{i}$ and $w_{i}$ both have
total degree $n-1$, we conclude that $w_{i}=\alpha c_{i}$ for some $\alpha
\in \mathbb{C}$. To prove that $\alpha =1$, it suffices to note that the
monomial $m_{i}=\prod_{j\neq i}a_{ij}$ occurs in both $c_{i}$ and $w_{i}$
with coefficient $1$.
\end{proof}

\section{The all minors theorem}

The \emph{all minors} theorem \cite{C} is a formula for $\det L_{IJ}$, where 
$L_{IJ}$ is the submatrix of $L$ obtained by deleting rows $I$ and columns $%
J $. It turns out this follows from Theorem \ref{matrix tree} by a
specialization of variables. We will state and prove this below after a
brief discussion on signs of permutations and bijections.

\subsection{Interlude on signs}

Let $I,J$ be equal-sized subsets of $\left\{ 1,\ldots ,n\right\} $ and let $%
\Sigma _{I},\Sigma _{J}$ denote the sums of their elements. If $\beta
:J\rightarrow I$ is a bijection, we write $inv\left( \beta \right) $ for the
number of inversions in $\beta $, \textit{i.e.} pairs $j<j^{\prime }$ in $J$
such that $\beta \left( j\right) >\beta \left( j^{\prime }\right) $ and we
define 
\begin{equation*}
\varepsilon \left( \beta \right) =\left( -1\right) ^{inv\left( \beta \right)
+\Sigma _{I}+\Sigma _{J}}.
\end{equation*}%
Note that if $J=I$ then $\varepsilon \left( \sigma \right) =\left( -1\right)
^{inv\left( \sigma \right) }$ is the sign of $\sigma $ as a permutation.

\begin{lemma}
\label{hij}If $\beta :J\rightarrow I$, $\alpha :I\rightarrow H$ are
bijections then $\varepsilon \left( \alpha \beta \right) =\varepsilon \left(
\alpha \right) \varepsilon \left( \beta \right) $.
\end{lemma}

\begin{proof}
This follows by combining the following mod $2$ congruences%
\begin{equation*}
\Sigma _{H}+\Sigma _{I}+\Sigma _{I}+\Sigma _{J}\equiv \Sigma _{H}+\Sigma _{J}%
\text{, }inv\left( \alpha \beta \right) \equiv inv\left( \alpha \right)
+inv\left( \beta \right) ,
\end{equation*}%
the first of which is obvious. To establish the second congruence we replace 
$\alpha ,\beta $ by the permutations $\lambda \alpha ,\beta \mu $ of $I$,
where $\lambda :H\rightarrow I,\mu :I\rightarrow J$ are the unique
order-preserving bijections; this does not affect $inv\left( \alpha \right) $
\textit{etc.,} and reduces the second congruence to a standard fact about
permutations.
\end{proof}

The meaning of $\varepsilon \left( \beta \right) $ is clarified by the
following result. For a bijection $\beta :J\rightarrow I$ and any $n\times n$
matrix $X$, let $X_{\beta }$ be the matrix obtained from $X$ by replacing,
for each $j\in J$, the $j$th column of $X$ by the unit vector $\mathbf{e}%
_{\beta \left( j\right) }.$

\begin{lemma}
\label{detXb}We have $\det X_{\beta }=\varepsilon \left( \beta \right) \det
X_{IJ}$.
\end{lemma}

\begin{proof}
If $\sigma $ is a permutation of $I$ then by the previous lemma, and
standard properties of the determinant, we have 
\begin{equation*}
\varepsilon \left( \sigma \beta \right) =\varepsilon \left( \sigma \right)
\varepsilon \left( \beta \right) \text{, }\det \left( X_{\sigma \beta
}\right) =\varepsilon \left( \sigma \right) \det \left( X_{\beta }\right) 
\end{equation*}%
Thus replacing $\beta $ by a suitable $\sigma \beta $, we may assume $%
inv\left( \beta \right) =0$ and write 
\begin{equation*}
I=\left\{ i_{1}<\cdots <i_{p}\right\} ,J=\left\{ j_{1}<\cdots <j_{p}\right\} 
\text{ with }\beta \left( j_{k}\right) =i_{k}\text{ for all }k\text{.}
\end{equation*}%
The lemma now follows from the identity 
\begin{equation*}
\det \left( X_{\beta }\right) =\left( -1\right) ^{i_{p}+j_{p}}\cdots \left(
-1\right) ^{i_{1}+j_{1}}\det X_{IJ}=\left( -1\right) ^{\Sigma _{I}+\Sigma
_{J}}\det X_{IJ}
\end{equation*}%
obtained by iteratively expanding $\det \left( X_{\beta }\right) $ along
columns $j_{p},\ldots ,j_{1}$.
\end{proof}

\subsection{Directed forests}

Let $\mathcal{F}\left( J\right) $ be the set of all directed spanning
forests $f$ of $G$ with root set $J$. Let $\mathcal{F}\subset \mathcal{F}%
\left( J\right) $ be the subset consisting of those forests $f$ such that
each tree of $f$ contains a unique vertex of $I$. Note that the trees of $%
f\in \mathcal{F}$ give a bijection $\beta _{f}:$ $J\rightarrow I$. The all
minors theorem is the following formula \cite{C}.

\begin{theorem}
We have $\det \left( L_{IJ}\right) =\sum_{f\in \mathcal{F}}\varepsilon
\left( \beta _{f}\right) \mathrm{wt}\left( f\right) $.
\end{theorem}

We fix a bijection $\beta :J\rightarrow I$ and define $\sigma _{f}=\beta
^{-1}\beta _{f}:J\rightarrow J$. In view of Lemmas \ref{hij} and \ref{detXb}%
, it suffices to prove the following reformulation of the previous theorem.

\begin{theorem}
We have $\det L_{\beta }=\sum_{f\in \mathcal{F}}\varepsilon \left( \sigma
_{f}\right) \mathrm{wt}\left( f\right) .$
\end{theorem}

\begin{proof}
As usual it is enough to treat the complete digraph $G_{n}$ with arbitrary
edge weights $a_{ij}$. We fix an index $j_{0}\in J$ and put $i_{0}=\beta
\left( j_{0}\right) $, $J_{0}=J\setminus \left\{ j_{0}\right\} $. We now
consider a particular specialization $\bar{a}_{ij}$ of $a_{ij}$, and the
entries $\bar{l}_{ij}$ of the specialized Laplacian $\bar{L}$. For $j\notin
J_{0}$ we set $\bar{a}_{ij}=a_{ij}$ and hence $\bar{l}_{ij}=a_{ij}$; while
for $j\in J_{0}$ we set 
\begin{equation}
\bar{a}_{ij}=\left\{ 
\begin{tabular}{cc}
$1$ & if $i=i_{0}$ \\ 
$-1$ & if $i=\beta \left( j\right) $ \\ 
$0$ & otherwise%
\end{tabular}%
\right. \implies \bar{l}_{ij}=\left\{ 
\begin{tabular}{cc}
$-1$ & if $i=i_{0}$ \\ 
$1$ & if $i=\beta \left( j\right) $ \\ 
$0$ & otherwise%
\end{tabular}%
\right.  \label{=sp}
\end{equation}%
Note that $\bar{L}$ and $L_{\beta }$ have the same entries outside of row $%
i_{0}$ and column $j_{0}$; hence we get $\det L_{\beta
}=c_{i_{0}j_{0}}\left( L_{\beta }\right) =c_{i_{0}j_{0}}\left( \bar{L}%
\right) $ and it remains to show that 
\begin{equation}
c_{i_{0}j_{0}}\left( \bar{L}\right) \overset{?}{=}\sum\nolimits_{f\in 
\mathcal{F}}\varepsilon \left( \sigma _{f}\right) \mathrm{wt}\left( f\right) 
\text{.}  \label{=cbar}
\end{equation}

Specializing Theorem \ref{matrix tree} we get 
\begin{equation*}
c_{i_{0}j_{0}}\left( \bar{L}\right) =\sum\nolimits_{f\in \mathcal{F}\left(
J\right) }\psi \left( f\right) \mathrm{wt}\left( f\right) \,\text{, }\psi
\left( f\right) :=\sum\nolimits_{t\in \mathcal{A}_{f}}\left( -1\right)
^{p\left( t\right) },
\end{equation*}%
where $\mathcal{A}_{f}$ is the set of $j_{0}$-trees $t$ such that for each $%
j\in J_{0}$ the unique edge $ij$ in $t$ satisfies $i=i_{0}$ or $i=\beta
\left( j\right) $, and for which deleting all such edges from $t$ yields the
forest $f$; and $p\left( t\right) $ is the number of edges in $t$ of type $%
i_{0}j$, $j\in J_{0}$. Therefore to prove (\ref{=cbar}) it suffices to show%
\begin{equation*}
\psi \left( f\right) \overset{?}{=}\left\{ 
\begin{tabular}{cc}
$0$ & $\text{if }f\notin \mathcal{F}$ \\ 
$\varepsilon \left( \sigma _{f}\right) $ & $\text{if }f\in \mathcal{F}$%
\end{tabular}%
\right. \text{.}
\end{equation*}

First suppose $f\notin \mathcal{F}$. In this case if $t\in \mathcal{A}_{f}$
then there is some $j\in J_{0}$ such that the $j$-subtree contains no $I$
vertex. Choose the largest such $j$ and change the edge $ij$, from $i=i_{0}$
to $i=\beta \left( j\right) $ or vice versa. This is a sign-reversing
involution on $\mathcal{A}_{f}$ and hence we get $\psi \left( f\right) =0$.

Now let $f\in \mathcal{F}$, and for each subset $S\subset J_{0}$ consider
the graph obtained from $f$ by adding the edges $i_{0}j$ for $j\in S$, and $%
\beta \left( j\right) j\ $for $j\in J_{0}\setminus S$. This graph is a tree
in $\mathcal{A}_{f}$ iff $S$ meets every cycle $c$ of the permutation $%
\sigma _{f}$ of $J$, and is disconnected otherwise. Thus a tree $t\in 
\mathcal{A}_{f}$ is prescribed uniquely by choosing, for each cycle $c$ of $%
\sigma _{f}$, a nonempty subset $S_{c}$ of its vertex set $J_{c}$. By
definition we have $\left( -1\right) ^{p\left( t\right)
}=\prod\nolimits_{c}\left( -1\right) ^{\left\vert J_{c}\right\vert
-\left\vert S_{c}\right\vert }$, and so $\psi \left( f\right) $ factors as 
\begin{equation*}
\psi \left( f\right) =\prod\nolimits_{c}\psi \left( c\right) \text{, }\psi
\left( c\right) :=\sum\nolimits_{J_{c}\supseteq S_{c}\neq \emptyset }\left(
-1\right) ^{\left\vert J_{c}\right\vert -\left\vert S_{c}\right\vert }.
\end{equation*}%
Now we get $\psi \left( c\right) =\left( -1\right) ^{\left\vert
J_{c}\right\vert -1}$ using the elementary identity 
\begin{equation*}
\sum\nolimits_{k=1}^{m}\binom{m}{k}\left( -1\right) ^{m-k}=\left( 1-1\right)
^{m}-\left( -1\right) ^{m}=\left( -1\right) ^{m-1}\text{.}
\end{equation*}%
Thus $\psi \left( f\right) $ agrees with the standard formula $%
\prod_{c}\left( -1\right) ^{\left\vert J_{c}\right\vert -1}$ for $%
\varepsilon \left( \sigma _{f}\right) $.
\end{proof}

\end{document}